\documentclass[10pt]{article}
\usepackage{palatino}
\usepackage{geometry} 
\usepackage{amsmath}
\usepackage{amssymb}
\usepackage{amsthm}
\usepackage{setspace}
\theoremstyle{plain}
\newtheorem{thm}{Theorem}
\newtheorem{lem}{Lemma}
\newtheorem{prop}{Proposition}
\newtheorem{cor}{Corollary}

\theoremstyle{definition}
\newtheorem{defn}{Definition}

\theoremstyle{remark}

\title{On linear codes and distinct weights  }
\author{Alessio Meneghetti}

\begin{document}
\maketitle
\abstract{
We provide a combinatorial construction for linear codes attaining the maximum possible number of distinct weights. We then introduce the related problem of determining the existence of linear codes with an arbitrary number of distinct non-zero weights, and we completely determine a solution in the binary case. 
}


\section*{Introduction}
A recent work on extremal properties of linear codes presented the problem of determining the maximum possible number of distinct non-zero weights in a linear code \cite{shi2018etal}, partially providing solutions and leaving the general case as a conjecture. For the sake of completeness, we introduce the problem and the conjecture.
\\
Let $C$ be an $[n,k]_q$ linear code and let us denote with $\mathcal{A}\left(C\right)$ the weight distribution of the code $C$. We also denote with $\mathrm{len}\left(C\right)$ the length of a code $C$ and with $\mathrm{dim}\left(C\right)$ its dimension.\\
We are interested in the maximum number of distinct weights that a linear code of dimension $k$ over $\mathbb{F}_q$ can have:
\begin{equation}\nonumber
\mathcal{L}_q\left(k\right):=\max\left\{\left|\mathcal{A}\left(C\right)\right|\;:\;\;\mathrm{dim}\left(C\right)=k\right\}
\end{equation}
We remark that this value is obtained by considering all possible codes of dimension $k$ over $\mathbb{F}_q$, regardless on their length. A related problem would be to determine the maximum number of weights in a code of length $n$, a value which we will denote as $\mathcal{L}_q\left(n,k\right)$. The relation between $\mathcal{L}_q\left(k\right)$ and $\mathcal{L}_q\left(n,k\right)$ is
\begin{equation}\nonumber
\mathcal{L}_q\left(k\right):=\max_n\mathcal{L}_q\left(n,k\right) ,
\end{equation}
and an interested reader can go through \cite{shi2018etal} for more results on this function.\\
By linearity it has to hold 
\begin{equation}\label{bound}
\mathcal{L}_q\left(k\right)\leq \frac{q^k-1}{q-1}+1 ,
\end{equation}
and in \cite{shi2018etal}
it is proved that $\mathcal{L}_2\left(k\right)\ge2^k$, and 
\begin{equation}\label{eq: case k2}
\mathcal{L}_q\left(k\right)\ge q+2 ,
\end{equation}
namely bound \eqref{bound} is achieved with equality by binary codes regardless of their dimension, and by $q$-ary codes of dimension $2$. In the same work this result was presented as a hint that Bound \eqref{bound} could actually be the value of $\mathcal{L}_q\left(k\right)$.

In \cite{alderson2018maximum}, the authors prove the conjecture both from an algebraic and a geometric point of view, and call Maximum Weight Spectrum (MWS) the codes with maximum number of distinct weights. They conclude by providing bounds on the minimum length of MWS codes.\\ 
We show here a third, combinatorial proof of the same conjecture, namely in Section \ref{sec: proof} we prove the following:
\begin{thm}\label{thm: Lqk}
$\mathcal{L}_q\left(k\right)= \frac{q^k-1}{q-1}+1 .$
\end{thm}
In the subsequent section, we discuss a related problem, the existence of linear codes accordingly to the number of distinct weights.


\section{Proof of Theorem \ref{thm: Lqk}}\label{sec: proof}
In this section we prove Theorem \ref{thm: Lqk}. The proof is by induction, where we use the results of \cite{shi2018etal} as a starting point for deducing the general case. \\
To ease the notation, we denote a q-ary code of dimension $k$ with $C_k$, its generator matrix with $G_k$, and we use $\mathrm{w}\left(S\right)$ for the set of distinct weights of a set of vectors $S$. Finally, we use $\bar{C}_k$ (and $\bar{G}_k$) for a code with the largest possible number of distinct weights, i.e. $\left|\mathrm{w}\left(\bar{C}_k\right)\right|=\mathcal{L}_q\left(k\right)$. We can thus state Theorem \ref{thm: Lqk} in a slightly alternative way.
\begin{thm}\label{thm: Lqk alternative}
For each $q$ and for each $k$ there exists an $[n,k]_q$ code $\bar{C}_k$ for which 
\begin{equation}\label{eq: number weights}
\left|\mathrm{w}\left(\bar{C}_k\right)\right|=\frac{q^k-1}{q-1}+1 .
\end{equation}
\end{thm}
The assumption of our proof is therefore that for all dimensions up to $k$, there exists a linear code $\bar{C}_k$ with $\frac{q^k-1}{q-1}+1$ distinct weights. 
\\
Starting by a code $\bar{C}_k$ we then obtain a new code $\bar{C}_{k+1}$ of dimension $k+1$ with $\left|\mathrm{w}\left(\bar{C}_{k+1}\right)\right|=\frac{q^{k+1}-1}{q-1}+1$.
\\
\begin{lem}\label{lem: case x exists}
Let $\bar{C}_k=\left\{c_1,\ldots,c_{q^k}\right\}$ satisfies Equation \eqref{eq: number weights}. If there exists $x\notin C$ such that $\mathrm{d}\left(c_i,x\right)\neq \mathrm{d}\left(c_j,x\right)$  for each pair of indices $i\neq j$, then there also exists a code $\bar{C}_{k+1}$ satisfying Equation \eqref{eq: number weights}.
\end{lem}
\begin{proof}
Let $m$ be the maximum weight of a codeword in $\bar{C}_k$. 
We construct $\bar{G}_{k+1}$ as the $\left(k+1\right)\times \left(\mathrm{len}\left(\bar{C}_k\right)+m\right)$ $q$-ary matrix in the following way:
\begin{equation}\nonumber
\bar{G}_{k+1}=
\begin{bmatrix}
\bar{G}_{k} & 0 \\
x & 1\cdots1
\end{bmatrix} .
\end{equation}
If we ignore the $m$ zeros at the end, the linear combinations of the first $k$ rows of $G_{k+1}$ are exactly the codewords of $\bar{C}_k$, hence we have at least $\frac{q^k-1}{q-1}+1$ distinct weights.\\
Any other codeword in $\bar{C}_{k+1}$ is a linear combination involving the last row of $\bar{G}_{k+1}$ and a codeword $c\in\bar{C}_k$. More precisely, any other codeword is of the form $\alpha y$, with $\alpha\neq 0$ an element of $\mathbb{F}_q$, and $y$ the concatenation of $c-x$ and a sequence of $m+1$ ones. 
Observe that $\mathrm{w}\left(\alpha y\right)=\mathrm{w}\left(y\right)$. \\
The weight of $y$ is equal to the sum between $m$ and the distance between $c$ and $x$. By hypothesis, $\mathrm{d}\left(c_1,x\right)\neq\mathrm{d}\left(c_2,x\right)$ for each $c_1\neq c_2\in\bar{C}_k$, hence we have $q^k$ distinct weights. Moreover, for any $y$ we have $\mathrm{w}\left(y\right)=m+\mathrm{d}\left(c,x\right)>m$. \\
We conclude that $\left|\mathrm{w}\left(C_{k+1}\right)\right|=\left|\mathrm{w}\left(C_{k}\right)\right|+q^k=\frac{q^{k+1}-1}{q-1}+1$.
\end{proof}
Lemma \ref{lem: case x exists} deals with a very particular case, still, in such restrictive hypotheses, it allows us to directly derive $\bar{C}_{k+1}$ from $\bar{C}_k$. However we cannot always assume the existence of the vector $x$ with the desired properties. To address the general case we will therefore also use Lemma \ref{lem: x generic}. \\
Let $C$ be a code and $x$ a vector of the same length of $C$. We denote with $C-x$ the coset of $C$ obtained by translating the codewords of $C$ by $-x$, namely, $C-x=\left\{c-x\;\mid c\in C\right\}$.
\begin{lem}\label{lem: x generic}
Let $\mathbb{F}_q$ be a finite field with at least $3$ elements, and let $C_k$ be a linear code of dimension $k$ over $\mathbb{F}_q$. Let $x\notin C_k$, with $\left|\mathrm{w}\left(C_k-x\right)\right|<q^k$. Then there exist $x'\notin C_k'$ with $\left|\mathrm{w}\left(C'_k-x'\right)\right|>\left|\mathrm{w}\left(C_k-x\right)\right|$.
\end{lem}
\begin{proof}
Given $C_k$ and $x$, if there exists $x'$ for which $\left|\mathrm{w}\left(C_k-x'\right)\right|>\left|\mathrm{w}\left(C_k-x\right)\right|$, we can conclude. We can therefore assume that no such vector $x'$ in $\left(\mathbb{F}_q\right)^{\mathrm{len}\left(C_k\right)}$ exists. 
\\
In this setting, we construct a new code $C_k'$ and a new vector $x^{\left(3\right)}$ outside of it by tripling each coordinate. The new code will have three times the length of $C_k$, and the distance between any two codewords will be three times larger. The same is true for the distance between $x^{\left(3\right)}$ and a codeword $c'$ of $C_k'$. In particular, if we consider two codewords $c_1$ and $c_2$ in $C_k$ such that $\mathrm{d}\left(c_1,x\right)=\mathrm{d}\left(c_2,x\right)$, then $\mathrm{d}\left(c'_1,x^{\left(3\right)}\right)=\mathrm{d}\left(c'_2,x^{\left(3\right)}\right)=3\mathrm{d}\left(c_1,x\right)$. We remark that the existence of such two codewords $c_1$ and $c_2$ is guaranteed by the hypotheses of the lemma. We denote $\mathrm{d}\left(c_1,x\right)$ with $t$, hence $\mathrm{d}\left(c'_1,x^{\left(3\right)}\right)$ and $\mathrm{d}\left(c'_2,x^{\left(3\right)}\right)$ are both equal to $3t$. Since $c_1'\neq c_2'$, there is at least a coordinate in which the two codewords are different, so without loss of generality we suppose they differ in the first coordinate. Let $\alpha$, $\beta$ and $\gamma\in\mathbb{F}_q$ be respectively the values of the first coordinates of the two codewords $c_1'$, $c_2'$ and of $x^{\left(3\right)}$. 
\\
We have two cases:
\begin{enumerate}
\item $\gamma$ is equal to either $\alpha$ or $\beta$;
\item $\gamma$ is not equal to $\alpha$ nor to $\beta$.
\end{enumerate} 
In the first case, let us say that $\gamma=\alpha$, namely both $c_1'$ and $x'$ are equal to $\alpha$ in the first coordinate, while $c_2$ have $\beta\neq \alpha$ in the first coordinate. We substitute now the first coordinate of $x^{\left(3\right)}$ with an element of $\mathbb{F}_q$ different from both $\alpha$ and $\beta$, so that the distance between $x^{\left(3\right)}$ and $c_1'$ would increase by $1$. We denote with $x'$ this modified vector. By computing the distances we obtain
$$
\begin{array}{l}
\mathrm{d}\left(c_1',x'\right)=3t+1\\
\mathrm{d}\left(c_2',x'\right)=3t.
\end{array}
$$
In the second case, $\gamma\neq \alpha\neq \beta\neq\gamma $. We change $\gamma$ into $\alpha$, and we obtain
$$
\begin{array}{l}
\mathrm{d}\left(c_1',x'\right)=3t-1\\
\mathrm{d}\left(c_2',x'\right)=3t.
\end{array}
$$
Either case, the distance between $c_1'$ and $c_2'$ did not change.
\\
To conclude, the proof that $\left|\mathrm{w}\left(C'_k-x'\right)\right|>\left|\mathrm{w}\left(C_k-x\right)\right|$ directly follows from the fact that for any codeword $c'\in C_k'$ we have $3\mid\mathrm{d}\left(c',x^{\left(3\right)}\right)$, and to obtain $x'$ from $x^{\left(3\right)}$ we changed a single coordinate.
\end{proof}
Observe that in the proof we are using the fact that $q\ge3$, so this proof cannot be directly applied to the binary case. However, we will not need it, since the binary case was already covered in  \cite{shi2018etal}. 
\begin{proof}[Proof of Theorem \ref{thm: Lqk}]
As already mentioned, we prove the theorem by induction, using Equation \eqref{eq: case k2} as initial step. Let $q>2$ and let $\bar{C}_k$ be a $q$-ary code as defined above, namely a code of dimension $k$ for which $\left|\mathrm{w}\left(\bar{C}_k\right)\right|=\frac{q^k-1}{q-1}+1$.  If there exists $x$ as in the hypotheses of Lemma \ref{lem: case x exists}, then we can construct a code $\bar{C}_{k+1}$. \\
Otherwise, we make use of Lemma \ref{lem: x generic}. We remark that if we start by $\bar{C}_k$, by tripling its coordinates we end up with a code $\bar{C}_k'$ which satisfy itself Equation \eqref{eq: number weights}. We keep applying Lemma \ref{lem: x generic}, and each time we increase the number of distinct weights of the coset $\left\{\bar{C}_k-x\right\}$. Since the number of elements are $q^2$, we eventually end up with $q^k$ distinct weights, and we can apply Lemma \ref{lem: case x exists}.
\end{proof}
The proposed proof provides a method to explicitly construct codes with the largest possible number of distinct weights. This method also gives us a coarse upper bound on the minimum length of such a code. In the worst-case scenario, Lemma \ref{lem: x generic} will be applied $q^{k}-1$ times before applying Lemma \ref{lem: case x exists}. Since the codeword of larger weight in $\bar{C}_k$ has at most $\mathrm{len}\left(\bar{C}_{k}\right)$ non-zero coordinates, then there exist $\bar{C}_{k+1}$ such that
$$
\mathrm{len}\left(\bar{C}_{k+1}\right)\leq \left(3^{q^k-1}+1\right)\mathrm{len}\left(\bar{C}_{k}\right).
$$


\section{Existence of codes with an arbitrary number of distinct weights}
In this section we look at the problem from a slightly different angle, asking ourselves whether it exist a linear $q$-ary code for a given number of distinct weights. The results presented in this section allow us to completely solve this problem in the binary case, though the general $q$-ary case still remains an open problem. 
\\
We denote with $\mathcal{I}_q$ the set of integers for which it exist a code over $\mathbb{F}_q$ with $s$ distinct non-zero weights. 
Observe that Theorem \ref{thm: Lqk} give us the largest $i$ in each $\mathcal{I}_q$, hence $\mathcal{I}_q\subseteq \left\{1,\ldots, \frac{q^k-1}{q-1}+1\right\}$.
\\
Moreover, we recall that for each field $\mathbb{F}_q$ and for each dimension $k$, it exist a linear equidistant code $\mathcal{S}_{q,k}$. We recall here its definition and parameters, an interested reader can go through \cite{huffman2010fundamentals} for a more deep understanding on the subject.
\begin{defn}
Let $\mathcal{G}_k$ be a $k\times \frac{q^k-1}{q-1}$ matrix over $\mathbb{F}_q$ with the property that its columns are pair-wise independent. We call Simplex code $\mathcal{S}_{q,k}$ the code generated by $G$.  
\end{defn}
\begin{prop}\label{prop: simplex}
$\mathcal{S}_{q,k}$ is an $[ \frac{q^k-1}{q-1},k, q^{k-1}]$ equidistant code, namely any pair of codewords are at distance $q^{r-1}$.
\end{prop}
Our interest in the class of Simplex codes is that their existence prove that for any $q$ and any $k$ there exist a code with a single non-zero weight. We will make use of these codes to prove the existence of codes with arbitrary distinct weights.
\begin{lem}\label{lem: induction small s}
Assume that there exist a $q$-ary linear code of dimension $k$ with $s$ distinct weights. Then, there exists a $q$-ary linear code of dimension $k+1$ with $s+1$ distinct weights. 
\end{lem}
\begin{proof}
First of all, notice that Proposition \ref{prop: simplex} implies the existence of $C_{k,1}$ for each dimension $k$.\\
We denote with $G_{k,s}$ the generator matrix of an $[n,k]_q$ code $C_{k,s}$ with $s$ distinct weights.\\
We consider now a matrix of the form
\begin{equation}\label{eq: Gk+1 small s}
G_{k+1,s+1}=\begin{bmatrix}
\begin{matrix}0\\G_{k,s}\end{matrix} & \mathcal{G}_{q,k+1}\\
\end{bmatrix}.
\end{equation}
Any non-zero linear combination of its rows can be written as $c=\left(c'\mid c''\right)$, where $c'$ belongs to the code generated by $G_{k,s}$, and $c''$ is a non-zero codeword of the Simplex code of dimension $k+1$. Observe that if $c'$ has a weight equal to $w$, then the weight of $c$ is $w+q^k$. Hence, $G_{k+1,s+1}$ generates a code $C_{k+1,s+1}$ of dimension $k+1$ and with $s+1$ distinct weights. 
\end{proof}
\begin{lem}\label{lem: induction large s}
Assume that there exists a $q$-ary linear code of dimension $k$ with $s$ distinct weights. Then, there exists a $q$-ary linear code of dimension $k+1$ with $q^k+s$ distinct weights. 
\end{lem}
\begin{proof}
Let $C$ be a code as in our hypothesis. Then, by applying the same argument as in the proof of Theorem \ref{thm: Lqk}, we obtain a code which proves our claim.
\end{proof}
By combining together the results of Lemmas \ref{lem: induction small s} and \ref{lem: induction large s} we cannot prove yet the general $q$-ary case. Indeed, we observe that, even if we assume the existence of codes with $s$ distinct weight for any dimension up to $k$ and for any $s$ in the range $\left\{1,\ldots,\frac{q^k-1}{q-1}\right\}$, we cannot construct codes of dimension $k+1$ corresponding to values of $s$ between $\frac{q^k-1}{q-1}$ and $q^k$. Still, we can focus on the binary case, and complete the characterisation in this particular case.
\begin{cor}
For any dimension $k$ and for any integer $s\in\left\{1,\ldots,2^k-1\right\}$, there exists a linear binary code with $s$ distinct non-zero weights.
\end{cor}
\begin{proof}
We proceed by induction on the dimension, using as initial step $k=1$. In this particular case we are dealing with codes containing only $2$ words, and the maximum number of non-zero weights is trivially $2^k-1=1$, i.e. any code of dimension $1$ proves the case $k=1$. We also observe that for $s=1$ we simply rely on the existence of the $k$-th dimensional Simplex Code.
\\
As inductive step, we assume that for any integer $s$ in the range $\left\{1,\ldots,2^k-1\right\}$ there exits a linear code of dimension $k$ with $s$ distinct weights. We now apply Lemma  \ref{lem: induction small s} to each possible $s$, hence we construct codes of dimension $k+1$ and $s\in\left\{2,\ldots,2^k\right\}$ distinct weights.  Now, we  apply instead Lemma \ref{lem: induction large s}, and obtain codes with $2^k+s$ distinct weights. Since $s\in\left\{1,\ldots,2^k-1\right\}$, by Lemma \ref{lem: induction large s} we obtain codes of dimension $k+1$ corresponding to integers in the set $\left\{2^k+1,\ldots,2^{k+1}-1\right\}$. Combining the two results, we prove the existence of codes of dimension $k+1$ and number of distinct non-zero weights in the range $\left\{1, \ldots, 2^k,2^{k}+1,\ldots,2^{k+1}\right\}$.
\end{proof}


\section{Conclusions}
We provided a combinatorial proof of the conjecture presented in \cite{shi2018etal} regarding the maximum number of distinct weights that a linear $q$-ary code can have. By using similar methods, we were also able to prove the existence of binary linear codes with any number of distinct weights. The general $q$-ary case is left here as a conjecture. Other than addressing the general case, future directions of this research would be to establish new bounds on the minimum length of codes with a given number of non-zero weights.


\bibliographystyle{amsalpha}
\bibliography{Refs}


\end{document}